\documentclass[11pt]{article}
\usepackage{hyperref}
\usepackage{times}  
\usepackage{mathpazo}
\usepackage{amssymb,amsmath,amsthm}
\usepackage{epsfig}
\usepackage{tcolorbox}
\usepackage{lipsum} 
\usepackage{enumitem}

\newtheorem{theorem}{Theorem}[section]
\newtheorem{proposition}[theorem]{Proposition}

\newtheorem{lemma}[theorem]{Lemma}

\newtheorem{corollary}[theorem]{Corollary}

\newtheorem{remark}[theorem]{Remark}

\newcommand{\qedsymb}{\hfill{\rule{2mm}{2mm}}}
\renewenvironment{proof}[1][]{\begin{trivlist}
\item[\hspace{\labelsep}{\bf\noindent Proof#1:\/}] }{\qedsymb\end{trivlist}}

\def\calZ{{\cal Z}}

\def\calC{{\cal C}}
\def\calA{{\cal A}}
\def\calB{{\cal B}}

\def\calE{{\cal E}}

\def\R{\mathbb{R}}

\newcommand\Prob[2]{{\Pr_{#1}\left[ {#2} \right]}}

\renewcommand{\epsilon}{\varepsilon}

\newcommand{\Fset}{\mathbb{F}}         

\begin{document}

\title{{\bf Larger Nearly Orthogonal Sets over Finite Fields}}

\author{
Ishay Haviv\thanks{School of Computer Science, The Academic College of Tel Aviv-Yaffo, Tel Aviv, Israel. Supported in part by the Israel Science Foundation (grant No.~1218/20).}
\and
Sam Mattheus\thanks{Department of Mathematics and Data Science, Vrije Universiteit Brussel, Brussel, Belgium. Supported by a postdoctoral fellowship 1267923N from the Research Foundation Flanders (FWO).}
\and
Aleksa Milojevi\'c\thanks{Department of Mathematics, ETH Z\"{u}rich, Z\"{u}rich, Switzerland. Email address: aleksa.milojevic@math.ethz.ch. Supported in part by SNSF grant 200021\_196965.}
\and
Yuval Wigderson\thanks{Institute for Theoretical Studies, ETH Z\"{u}rich, Z\"{u}rich, Switzerland. Email address: yuval.wigderson@eth-its.ethz.ch. Supported by Dr. Max R\"{o}ssler, the Walter Haefner Foundation, and the ETH Z\"{u}rich Foundation.}
}

\date{}

\maketitle

\begin{abstract}
For a field $\Fset$ and integers $d$ and $k$, a set $\calA \subseteq \Fset^d$ is called $k$-nearly orthogonal if its members are non-self-orthogonal and every $k+1$ vectors of $\calA$ include an orthogonal pair. We prove that for every prime $p$ there exists some $\delta = \delta(p)>0$, such that for every field $\Fset$ of characteristic $p$ and for all integers $k \geq 2$ and $d \geq k$, there exists a $k$-nearly orthogonal set of at least $d^{\delta \cdot k/\log k}$ vectors of $\Fset^d$. The size of the set is optimal up to the $\log k$ term in the exponent.
We further prove two extensions of this result.
In the first, we provide a large set $\calA$ of non-self-orthogonal vectors of $\Fset^d$ such that for every two subsets of $\calA$ of size $k+1$ each, some vector of one of the subsets is orthogonal to some vector of the other.
In the second extension, every $k+1$ vectors of the produced set $\calA$ include $\ell+1$ pairwise orthogonal vectors for an arbitrary fixed integer $1 \leq \ell \leq k$.
The proofs involve probabilistic and spectral arguments and the hypergraph container method.
\end{abstract}

\section{Introduction}

For a field $\Fset$ and an integer $d$, two vectors $u,v \in \Fset^d$ are called orthogonal if they satisfy $\langle u,v \rangle = 0$ with respect to the standard inner product defined by $\langle u,v \rangle = \sum_{i=1}^{d}{u_i \cdot v_i}$. A vector $u \in \Fset^d$ is called self-orthogonal if $\langle u,u \rangle = 0$, and it is called non-self-orthogonal otherwise.
For integers $k$ and $\ell$ with $k \geq \ell$, a set $\calA \subseteq \Fset^d$ is said to be $(k,\ell)$-nearly orthogonal if its vectors are non-self-orthogonal and any set of $k+1$ members of $\calA$ includes $\ell+1$ pairwise orthogonal vectors. Let $\alpha(d,k,\ell,\Fset)$ denote the largest possible size of a $(k,\ell)$-nearly orthogonal subset of $\Fset^d$. For the special case of $\ell=1$, we refer to a $(k,1)$-nearly orthogonal set as $k$-nearly orthogonal, and we let $\alpha(d,k,\Fset) = \alpha(d,k,1,\Fset)$.
Note that for a field $\Fset$ and an integer $d$, $\alpha(d,1,\Fset)$ is the largest possible size of a set of non-self-orthogonal vectors in $\Fset^d$ that are pairwise orthogonal, hence $\alpha(d,1,\Fset) = d$.

A simple upper bound on $\alpha(d,k,\Fset)$ stems from Ramsey theory.
To see this, consider a $k$-nearly orthogonal set $\calA \subseteq \Fset^d$, and let $G$ denote the graph on the vertex set $\calA$, in which two vertices are adjacent if and only if their vectors are orthogonal. Since the vectors of $\calA$ are non-self-orthogonal and lie in $\Fset^d$, the graph $G$ has no clique of size $d+1$. Since every $k+1$ members of $\calA$ include an orthogonal pair, the graph $G$ has no independent set of size $k+1$. It thus follows that the size of $\calA$ is smaller than the Ramsey number $R(d+1,k+1)$. Using the upper bound on Ramsey numbers of Erd{\H{o}}s and Szekeres~\cite{ErdosS35}, it follows that $\alpha(d,k,\Fset) < \binom{d+k}{k}$, so in particular, we have $\alpha(d,k,\Fset) \leq O(d^k)$ for every fixed integer $k$. A poly-logarithmic improvement follows from the upper bound on Ramsey numbers due to Ajtai, Koml{\'{o}}s, and Szemer{\'{e}}di~\cite{AjtaiKS80}.

The problem of determining the values of $\alpha(d,k,\Fset)$ where $\Fset$ is the real field $\R$ was suggested by Erd{\H{o}}s in the late 1980s (see~\cite{NR97Erdos}).
By considering a set that consists of the vectors of $k$ pairwise disjoint orthogonal bases of $\R^d$, it follows that $\alpha(d,k,\R) \geq k \cdot d$.
Rosenfeld~\cite{Rosenfeld91} proved that this bound is tight for $k=2$, and F{\"{u}}redi and Stanley~\cite{FurediS92} showed that $\alpha(4,5,\R) \geq 24$, which implies that it is not tight in general. They further showed that for every fixed integers $d$ and $\ell$, the limit $\lim_{k \rightarrow \infty}\alpha(d,k,\ell,\R)/k$ exists and grows exponentially in $d$.
Alon and Szegedy~\cite{AlonS99} proved that for every integer $\ell \geq 1$ there exists a constant $\delta = \delta(\ell) > 0$, such that for all integers $d$ and $k \geq \ell$ with $k \geq 3$, it holds that
\begin{eqnarray}\label{eq:AS}
\alpha(d,k,\ell,\R) \geq d^{\delta \cdot \log k / \log \log k},
\end{eqnarray}
where here and throughout the paper, all logarithms are in base $2$.
On the upper bound side, Balla, Letzter, and Sudakov~\cite{BallaLS20} proved that $\alpha(d,k,\R) \leq O(d^{(k+1)/3})$ for every fixed integer $k$, improving on the $O(d^k)$ bound that follows from the Erd{\H{o}}s--Szekeres bound. Yet, the known lower and upper bounds on $\alpha(d,k,\R)$ for general values of $d$ and $k$ are rather far apart.

In a recent paper, Balla~\cite{Balla23} considered a bipartite variant of the notion of nearly orthogonal sets, giving rise to the following definition. For a field $\Fset$ and integers $d$ and $k$, let $\beta(d,k,\Fset)$ denote the largest possible size of a set $\calA \subseteq \Fset^d$ of non-self-orthogonal vectors, such that for every two (not necessarily disjoint) sets $A_1, A_2 \subseteq \calA$ of size $k+1$ each, there exist vectors $v_1 \in A_1$ and $v_2 \in A_2$ with $\langle v_1, v_2 \rangle = 0$.
Since such a set $\calA$ is $k$-nearly orthogonal, it follows that $\alpha(d,k,\Fset) \geq \beta(d,k,\Fset)$. It was proved in~\cite{Balla23} that there exists a constant $\delta >0$, such that for all integers $d$ and $k \geq 3$, it holds that $\beta(d,k,\R) \geq d^{\delta \cdot \log k/ \log \log k}$. This strengthens the result given in~\eqref{eq:AS} for the case $\ell=1$.

The study of nearly orthogonal sets over finite fields was proposed by Codenotti, Pudl{\'{a}}k, and Resta~\cite{CodenottiPR00}.
Motivated by questions in circuit complexity, they explored the quantity $\alpha(d,2,\Fset_2)$, which in turn, attracted further attention in the area of information theory (see, e.g.,~\cite{BargZ22,BlasiakKL13,ChawinH24}).
In striking contrast to the real field~\cite{Rosenfeld91}, it was shown in~\cite{GolovnevH20} that there exists a constant $\delta>0$ such that $\alpha(d,2,\Fset_2) \geq d^{1+\delta}$ for infinitely many integers $d$. It was recently shown in~\cite{ChawinH24} that for every prime $p$ there exists a constant $\delta = \delta(p)>0$, such that for every field $\Fset$ of characteristic $p$ and for all integers $k \geq 2$ and $d \geq k^{1/(p-1)}$, it holds that $\beta(d,k,\Fset) \geq d^{\delta \cdot k^{1/(p-1)}/\log k}$. In particular, for the binary field, it follows that $\alpha(d,k,\Fset_2) \geq \beta(d,k,\Fset_2) \geq d^{\Omega ( k/\log k)}$, and this is tight up to the $\log k$ term in the exponent.

\subsection{Our Contribution}

In the present paper, we prove lower bounds on $\alpha(d,k,\ell,\Fset)$ and $\beta(d,k,\Fset)$ for fields $\Fset$ of finite characteristic.
The following theorem improves the aforementioned result of~\cite{ChawinH24} for all fields of finite characteristic at least $3$.

\begin{theorem}\label{thm:IntroBeta}
For every prime $p$, there exists a constant $\delta = \delta(p) > 0$, such that for every field $\Fset$ of characteristic $p$ and for all integers $k \geq 2$ and $d \geq k$, it holds that
\[ \beta(d,k,\Fset) \geq d^{\delta \cdot k/\log k}.\]
\end{theorem}
\noindent
Note that the condition $d \geq k$ in Theorem~\ref{thm:IntroBeta} is essential, in the sense that for arbitrary integers $d$ and $k$, the bound guaranteed by the theorem might exceed the number of vectors in $\Fset^d$.

Recalling that $\alpha(d,k,\Fset) \geq \beta(d,k,\Fset)$, the bound stated in Theorem~\ref{thm:IntroBeta} for $\beta(d,k,\Fset)$ holds for $\alpha(d,k,\Fset)$ as well.
The following theorem extends this implication to the quantities $\alpha(d,k,\ell,\Fset)$ for an arbitrary fixed integer $\ell \geq 1$.

\begin{theorem}\label{thm:IntroAlpha}
For every prime $p$ and every integer $\ell \geq 1$, there exists a constant $\delta = \delta(p,\ell) > 0$, such that for every field $\Fset$ of characteristic $p$ and for all integers $k \geq 2$ and $d \geq k \geq \ell$, it holds that
\[ \alpha(d,k,\ell,\Fset) \geq d^{\delta \cdot k/\log k}.\]
\end{theorem}

We remark that the bounds in Theorems~\ref{thm:IntroBeta} and~\ref{thm:IntroAlpha} are meaningful for sufficiently large values of $k$, because constant values of $k$ can be handled by an appropriate choice of the constant $\delta$, using a trivial lower bound of $d$.
We further note that the bounds in both theorems depend on the characteristic $p$ of the field $\Fset$ rather than on its size. It thus suffices to consider in their proofs the finite field $\Fset_p$ of order $p$, which forms a sub-field of any field of characteristic $p$.

The proofs of Theorems~\ref{thm:IntroBeta} and~\ref{thm:IntroAlpha} rely on the probabilistic approach of Alon and Szegedy~\cite{AlonS99} in their construction of large nearly orthogonal sets over the reals (see also~\cite{Balla23,ChawinH24}).
The main novel ingredients, compared to~\cite{ChawinH24} and potentially of independent interest, are estimations for the number of subgraphs of certain types in pseudo-random graphs (specifically, regular graphs with a large spectral gap).
For Theorem~\ref{thm:IntroBeta}, we prove an upper bound on the number of bounded-size bi-independent sets, i.e., pairs of sets of vertices with no edge connecting a vertex of one set to a vertex of the other (see Theorem~\ref{thm:ARbipartit}). The proof of this result adapts a technique of Alon and R{\"{o}}dl~\cite{AlonR05} for counting independent sets in pseudo-random graphs.
For Theorem~\ref{thm:IntroAlpha}, we prove an upper bound on the number of bounded-size subgraphs that contain no copy of the complete graph of some fixed order (see Theorem~\ref{thm:F-free}). The proof incorporates the hypergraph container method, developed independently by Balogh, Morris, and Samotij~\cite{BaloghMS15} and by Saxton and Thomason~\cite{SaxtonT2015}, and a result of Alon on the number of copies of a fixed graph in pseudo-random graphs (see~\cite{Pseudo06}).
To establish Theorems~\ref{thm:IntroBeta} and~\ref{thm:IntroAlpha}, we apply these results to an appropriate family of graphs, termed orthogonality graphs and studied in~\cite{AlonK97,Vinh08a}, and combine the obtained bounds with the technique of~\cite{AlonS99}.
In fact, for convenience of presentation, we prove the existence of a set of vectors that simultaneously yields the bounds stated in both theorems (see Theorem~\ref{thm:General} and the subsequent paragraph).

We finally mention that our results provide, for every field $\Fset$ of finite characteristic, $k$-nearly orthogonal sets over $\Fset$ whose size is optimal up to the $\log k$ term in the exponent.
As noted earlier, over $\R$ there is a more substantial gap between the known lower and upper bounds.
It would be interesting to narrow the gaps in both cases.

\section{Counting Subgraphs of Pseudo-random Graphs}\label{sec:2}

In this section, we prove our results on counting subgraphs of pseudo-random graphs.
We start with a brief introduction to the concept of $(n,d,\lambda)$-graphs.

\subsection{Pseudo-random Graphs}

An $(n,d,\lambda)$-graph is a $d$-regular graph on $n$ vertices, such that the absolute value of every eigenvalue of its adjacency matrix, besides the largest one, is at most $\lambda$. Throughout the paper, the graphs may have loops, at most one at each vertex, where a loop contributes $1$ to the degree of its vertex.
It is well known that $(n,d,\lambda)$-graphs with $\lambda$ significantly smaller than $d$ enjoy strong pseudo-random properties and behave, in various senses, like a random graph with $n$ vertices and edge probability $d/n$. For a thorough introduction to the topic, the reader is referred to~\cite{Pseudo06}.

We state below two results on $(n,d,\lambda)$-graphs.
The first is the following lemma given in~\cite{AlonR05}.

\begin{lemma}[{\cite[Lemma~2.2]{AlonR05}}]\label{lemma:N(u)}
Let $G=(V,E)$ be an $(n,d,\lambda)$-graph, and let $B \subseteq V$ be a set of vertices.
Define
\[C = \Big \{ u \in V ~\Big{|}~ |N(u) \cap B| \leq \frac{d}{2n} \cdot |B| \Big \},\] where $N(u)$ denotes the set of neighbors of $u$ in $G$ (including $u$ itself, if there is a loop at $u$). Then
\[|B| \cdot |C| \leq  \Big (\frac{2 \lambda n}{d} \Big )^2.\]
\end{lemma}

The second result that we state here was proved by Alon (see~\cite{Pseudo06}).
Here, for a graph $F$, we denote its maximum degree by $\Delta(F)$, its automorphism group by $\mathrm{Aut}(F)$, and the number of its edges by $e(F)$.
For a graph $G$ and a subset $U$ of its vertex set, $G[U]$ stands for the subgraph of $G$ induced by $U$.

\begin{theorem}[{\cite[Theorem~4.10]{Pseudo06}}]\label{thm:AlonSubgraph}
Let $G = (V,E)$ be an $(n,d,\lambda)$-graph with, say, $d \leq 0.9 \cdot n$.
Let $F$ be a fixed graph on $\ell$ vertices, and let $u \leq n$ satisfy $u = \omega(\lambda \cdot (\frac{n}{d})^{\Delta(F)})$.
Then, for every set $U \subseteq V$ of size $u$, the number of (not necessarily induced) copies of $F$ in $G[U]$ is $(1+o(1)) \cdot \frac{u^{\ell}}{|\mathrm{Aut}(F)|} \cdot (\frac{d}{n})^{e(F)}$.
\end{theorem}

\begin{remark}\label{remark:G_n}
Strictly speaking, $G$ represents in Theorem~\ref{thm:AlonSubgraph} an infinite sequence $(G_n)$ of graphs, where $G_n$ has $n$ vertices for each $n$, and the $o(\cdot)$ and $\omega(\cdot)$ notations are used with respect to $n$ that tends to infinity. The same convention will be used in Theorem~\ref{thm:F-free}.
\end{remark}

\subsection{Bi-independent Sets}

We prove the following bipartite analogue of a result of Alon and R{\"{o}}dl~\cite{AlonR05}, established through a similar argument.

\begin{theorem}\label{thm:ARbipartit}
Let $G=(V,E)$ be an $(n,d,\lambda)$-graph, and let $s = \frac{2n \log n}{d}$.
Then for every integer $k \geq s$, the number of pairs $(U_1,U_2)$ of (not necessarily disjoint) subsets of $V$ with $|U_1| = |U_2| = k$, such that no edge of $G$ connects a vertex of $U_1$ to a vertex of $U_2$, is at most
\[\frac{1}{k!} \cdot n^{2s} \cdot \Big ( \frac{2 \lambda n}{d} \Big)^{2 \cdot (k-s)}.\]
\end{theorem}

\begin{proof}
Consider the sequences $u_1, v_1, u_2, v_2, \ldots, u_k, v_k$ of $2k$ vertices of $G$, such that the vertices $u_1, \ldots, u_k$ are distinct, the vertices $v_1, \ldots, v_k$ are distinct, and no edge of $G$ connects a vertex of $\{u_1, \ldots, u_k\}$ to a vertex of $\{v_1, \ldots, v_k\}$.
Such a sequence can be chosen in $k$ iterations, where the $i$th iteration, $0 \leq i <k$, is dedicated to choosing $u_{i+1}$ and $v_{i+1}$.
Let $B_0 = V$, and for each $i \in [k-1]$, let $B_i$ denote the set of all vertices of $G$ that are not adjacent to any of the vertices of $\{u_1, \ldots, u_i\}$.
Note that $B_i = B_{i-1} \setminus N(u_i) \subseteq B_{i-1}$.
We further define
\[C_i = \Big \{u \in V ~\Big{|}~ |N(u) \cap B_i| \leq \frac{d}{2n} \cdot |B_i| \Big \}\]
and apply Lemma~\ref{lemma:N(u)} to obtain that $|B_i| \cdot |C_i| \leq (\frac{2 \lambda n}{d})^2$.

Suppose that we have already chosen the first $2i$ vertices $u_1, v_1, \ldots, u_i, v_i$, and consider the choice of $u_{i+1}$ and $v_{i+1}$.
Since $v_{i+1}$ is not allowed to be adjacent to the vertices of $\{u_1, \ldots, u_i\}$, it must be chosen from $B_i$.
Further, if $u_{i+1}$ is not chosen from $C_i$, then $|N(u_{i+1}) \cap B_i| > \frac{d}{2n} \cdot |B_i|$, and thus $|B_{i+1}| < (1-\frac{d}{2n}) \cdot |B_i|$.
Therefore, for at most $s = \frac{2n \log n}{d}$ of the indices $i$, it holds that $u_{i+1} \notin C_i$.
On the other hand, if $u_{i+1}$ is chosen from $C_i$, then the number of possibilities to choose $u_{i+1}$ and $v_{i+1}$ is at most $|B_i| \cdot |C_i| \leq (\frac{2 \lambda n}{d})^2$.
It thus follows that the number of ways to choose the sequence $u_1, v_1, \ldots, u_k, v_k$ does not exceed
\[ \binom{k}{s} \cdot n^{2s} \cdot \Big ( \frac{2 \lambda n}{d}\Big )^{2 \cdot (k-s)}.\]
Indeed, there are $\binom{k}{s}$ ways to choose $s$ indices covering all the indices $i$ with $u_{i+1} \notin C_i$, and for each such index, there are at most $n^2$ ways to choose $u_{i+1}$ and $v_{i+1}$. As shown above, for each of the remaining $k-s$ indices $i$, there are at most $(\frac{2 \lambda n}{d})^2$ ways to choose $u_{i+1}$ and $v_{i+1}$.
We finally divide the obtained bound by $(k !)^2$, to avoid counting the permutations of the vertices of $\{u_1, \ldots, u_k\}$ and of $\{v_1, \ldots, v_k\}$. This yields the desired bound and completes the proof.
\end{proof}

We derive the following corollary.

\begin{corollary}\label{cor:ARbipartite<=}
Let $G=(V,E)$ be an $(n,d,\lambda)$-graph, and let $s = \frac{2n \log n}{d}$.
Then for every integer $k$, the number of pairs $(U_1,U_2)$ of (not necessarily disjoint) subsets of $V$ with $|U_1| \leq k$ and $|U_2| \leq k$, such that no edge of $G$ connects a vertex of $U_1$ to a vertex of $U_2$, is at most
\[(k+1)^2 \cdot \max \bigg ( n, \Big (\frac{2 \lambda n}{d} \Big )^2 \bigg )^{s+k}.\]
\end{corollary}

\begin{proof}
For a given integer $k$ and for arbitrary integers $0 \leq k_1, k_2 \leq k$, consider the pairs $(U_1,U_2)$ of subsets of $V$ with $|U_1| = k_1$ and $|U_2| = k_2$, such that no edge of $G$ connects a vertex of $U_1$ to a vertex of $U_2$.
Suppose without loss of generality that $k_1 \leq k_2$.
If $k_1 < s$, then the number of these pairs is clearly bounded by $n^{k_1+k_2} < n^{s+k}$.
Otherwise, by Theorem~\ref{thm:ARbipartit}, there are at most
\[n^{2s} \cdot \Big (\frac{2 \lambda n}{d} \Big )^{2 \cdot (k_1-s)}\]
ways to choose $k_1$ vertices for each of $U_1$ and $U_2$, and there are at most $n^{k_2-k_1}$ ways to choose additional $k_2-k_1$ vertices for $U_2$. Therefore, the number of pairs in this case does not exceed
\[n^{2s} \cdot \Big (\frac{2 \lambda n}{d} \Big )^{2 \cdot (k_1-s)} \cdot n^{k_2-k_1} \leq \max \bigg ( n, \Big (\frac{2 \lambda n}{d} \Big )^2 \bigg )^{2s+(k_1-s)+(k_2-k_1)} \leq \max \bigg ( n, \Big (\frac{2 \lambda n}{d} \Big )^2 \bigg )^{s+k}.\]
The proof is concluded by considering all the possible values of the integers $k_1$ and $k_2$.
\end{proof}

\subsection{Subgraphs Free of Small Complete Graphs}

We prove the following theorem (see Remark~\ref{remark:G_n}).
As usual, $K_\ell$ denotes the complete graph of order $\ell$.

\begin{theorem}\label{thm:F-free}
Let $G=(V,E)$ be an $(n,d,\lambda)$-graph with, say, $d \leq 0.9 \cdot n$.
Suppose that $n = \Theta(d)$ and $n = \omega(\lambda)$.
Then, for every fixed integer $\ell \geq 2$ there exists a constant $c$, such that for all integers $k \leq n$, the number of sets $U \subseteq V$ of size at most $k$ for which $G[U]$ contains no copy of $K_\ell$ is at most
\[2^{c \cdot \log n \cdot \log(\frac{n}{\lambda})} \cdot \lambda^k.\]
\end{theorem}

\paragraph*{The Container Method.}

In what follows, we present a statement of the container method, as given by Saxton and Thomason in~\cite{SaxtonT16}.
We start with some notations.
For an integer $\ell \geq 2$, let $H$ be an $\ell$-uniform hypergraph on the vertex set $V$.
Let $P(V)$ denote the power set of $V$, and let $e(H)$ denote the number of hyperedges in $H$.
For a set $U \subseteq V$, let $H[U]$ denote the sub-hypergraph of $H$ induced by $U$.
The set $U$ is called an independent set of $H$ if $e(H[U])=0$.
For a set $\sigma \subseteq V$ of size $|\sigma| \leq \ell$, let $d(\sigma)$ denote the number of hyperedges in $H$ that contain $\sigma$.
For each $2 \leq j \leq \ell$ and for every vertex $v \in V$, let $d^{(j)}(v)$ denote the maximum of $d(\sigma)$ over all sets $\sigma \subseteq V$ with $|\sigma| = j$ and $v \in \sigma$. It clearly holds that $d^{(j)}(v) \leq |V|^{\ell-j}$.
For each $2 \leq j \leq \ell$ and for any real $\tau > 0$, we define $\delta_j(H,\tau) = \frac{1}{\tau^{j-1} \cdot \ell \cdot e(H)} \cdot \sum_{v \in V}{d^{(j)}(v)}$ and $\delta(H,\tau) = 2^{\binom{\ell}{2}-1} \cdot \sum_{j=2}^{\ell}{2^{-\binom{j-1}{2}} \cdot \delta_j(H,\tau)}$.

The following theorem forms a simplified version of~\cite[Theorem~5.1]{SaxtonT16}.

\begin{theorem}[\cite{SaxtonT16}]\label{thm:container}
For a fixed integer $\ell \geq 2$, let $H$ be an $\ell$-uniform hypergraph on the vertex set $V$, and let $e_0$ be an integer satisfying $e_0 \leq e(H)$.
Let $\tau: P(V) \rightarrow \R^+$ be a function such that for every set $U \subseteq V$ with $e(H[U]) \geq e_0$, it holds that
\[\tau(U) < \frac{1}{2}~~~~\mbox{ and }~~~~\delta(H[U],\tau(U)) \leq \frac{1}{12 \cdot \ell !}.\]
Define
\[f_0 = \max \{ -|U| \cdot \tau(U) \cdot \log \tau(U) \mid U \subseteq V,~ e(H[U]) \geq e_0 \}.\]
Then there exists a collection $\calC \subseteq P(V)$, such that
\begin{enumerate}
  \item every independent set of $H$ is contained in some set of $\calC$,
  \item $e(H[C]) \leq e_0$ for each $C \in \calC$, and
  \item $\log |\calC| \leq O(f_0 \cdot \log (\frac{e(H)}{e_0}))$.
\end{enumerate}
\end{theorem}

Equipped with Theorem~\ref{thm:container}, we are ready to prove Theorem~\ref{thm:F-free}.

\begin{proof}[ of Theorem~\ref{thm:F-free}]
Fix an integer $\ell \geq 2$, and let $H$ denote the $\ell$-uniform hypergraph on the vertex set of $G$, where a set $U$ of $\ell$ vertices forms a hyperedge in $H$ if and only if $G[U]$ is a copy of $K_\ell$ in $G$.
Applying Theorem~\ref{thm:AlonSubgraph} with $F$ being $K_\ell$, using $n = \Theta(d)$ and $n = \omega(\lambda)$, we obtain that the number of hyperedges of $H$ satisfies $e(H) = \Theta(n^{\ell})$.
For a given integer $k \leq n$, our goal is to prove that the number of independent sets in $H$ of size at most $k$ is bounded by $2^{O(\log n \cdot \log(\frac{n}{\lambda}))} \cdot \lambda^k$.
In fact, it suffices to prove such a bound on the number of independent sets in $H$ of size exactly $k$.
This indeed follows using the fact that $\lambda \geq \Omega(\sqrt{d})$ (see, e.g.,~\cite{KrivelevichS03}).

We apply the container method, described in Theorem~\ref{thm:container}.
Define, say, $e_0 = \lambda^{\ell} \cdot \log (\frac{n}{\lambda})$, and notice that the assumption $n = \omega(\lambda)$ implies that $e_0 = \omega(\lambda^{\ell})$ and $e_0 \leq e(H)$ (for a sufficiently large $n$).
For a set $U \subseteq V$ of size $u$, consider the hypergraph $H[U]$, denote $m = e(H[U])$, and suppose that $m \geq e_0$.
This obviously implies that $u \geq e_0^{1/\ell} = \omega(\lambda)$, hence using $n = \Theta(d)$, we can apply Theorem~\ref{thm:AlonSubgraph} to obtain that $m = \Theta(u^{\ell})$.
For each $2 \leq j \leq \ell$, every vertex $v \in U$ satisfies in $H[U]$ that $d^{(j)}(v) \leq u^{\ell-j}$, hence for any $\tau >0$,
\[\delta_j(H[U],\tau) \leq \frac{u \cdot u^{\ell-j}}{\tau^{j-1} \cdot \ell \cdot m} \leq O \bigg ( \frac{1}{(\tau \cdot u)^{j-1}}\bigg ).\]
Setting $\tau(U) = \frac{a}{u}$ for a sufficiently large constant $a$, it holds that
\[\delta(H[U],\tau(U)) \leq O \Big ( \sum_{j=2}^{\ell}{\delta_j(H[U],\tau(U))} \Big ) \leq \frac{1}{12 \cdot \ell !}.\]
For a growing $n$, using $u = \omega(\lambda)$, it further follows that $\tau(U) < 1/2$.
We also observe that the quantity $f_0$ from Theorem~\ref{thm:container} satisfies $f_0 \leq O(\log n)$.
Indeed, for every set $U \subseteq V$, our definition of $\tau$ implies that $- |U| \cdot \tau(U) \cdot \log \tau(U) \leq O (\log |U|) \leq O(\log n)$.
We finally notice, using $e(H) = \Theta(n^\ell)$ and $e_0 \geq \lambda^\ell$, that $\log(\frac{e(H)}{e_0}) \leq O(\log(\frac{n}{\lambda}))$.

Now, we derive from Theorem~\ref{thm:container} that there exists a collection $\calC \subseteq P(V)$, such that
\begin{enumerate}
  \item\label{itm:1} every independent set of $H$ is contained in some set of $\calC$,
  \item\label{itm:2} $e(H[C]) \leq e_0$ for each $C \in \calC$, and
  \item\label{itm:3} $\log |\calC| \leq O( f_0 \cdot \log(\frac{e(H)}{e_0})) \leq O(\log n \cdot \log(\frac{n}{\lambda}))$.
\end{enumerate}
By Theorem~\ref{thm:AlonSubgraph}, there exists a constant $c_0$, such that every set $U \subseteq V$ with $|U| > c_0 \cdot e_0^{1/\ell} = \omega(\lambda)$ satisfies $e(H[U]) = \Theta(|U|^\ell) > e_0$.
Hence, Item~\ref{itm:2} above implies that $|C| \leq c_0 \cdot e_0^{1/\ell}$ for each $C \in \calC$.
It therefore follows from Item~\ref{itm:1} that the number of independent sets of $H$ of size $k$ does not exceed
\begin{eqnarray*}
|\calC| \cdot \binom{c_0 \cdot e_0^{1/\ell}}{k} &\leq & 2^{O(\log n \cdot \log(\frac{n}{\lambda}))} \cdot \Big (\frac{c_0 \cdot e_0^{1/\ell} \cdot e}{k} \Big )^k \\
&\leq& 2^{O(\log n \cdot \log(\frac{n}{\lambda}))} \cdot \lambda^{k} \cdot \Big ( \frac{c_0 \cdot \log^{1/\ell}(\frac{n}{\lambda}) \cdot e}{k} \Big )^k \\
&\leq& 2^{O(\log n \cdot \log(\frac{n}{\lambda}))} \cdot \lambda^{k}.
\end{eqnarray*}
Here, the first inequality follows by Item~\ref{itm:3} and the inequality $\binom{n}{k} \leq (\frac{n \cdot e}{k})^k$, and the second by the definition of $e_0$. For the third inequality, notice that the term $( \frac{c_0 \cdot \log^{1/\ell}(\frac{n}{\lambda}) \cdot e}{k} )^k$ is bounded from above by $1$ for $k \geq c_0 \cdot \log^{1/\ell}(\frac{n}{\lambda}) \cdot e$, and by $2^{O(\log n \cdot \log(\frac{n}{\lambda}))}$ for any other $k$.
This completes the proof.
\end{proof}

\section{Nearly Orthogonal Sets over Finite Fields}\label{sec:3}

In this section, we establish the following theorem.

\begin{theorem}\label{thm:General}
For every prime $p$ and every integer $\ell \geq 2$, there exists a constant $\delta = \delta(p,\ell) >0$, such that for all integers $k \geq 2$ and $d \geq k \geq \ell$, the following holds.
There exists a set $\calA$ of at least $d^{\delta \cdot k/\log k}$ non-self-orthogonal vectors of $\Fset_p^d$, such that
\begin{enumerate}
  \item\label{itm:22} every set $A \subseteq \calA$ with $|A|=k$ includes $\ell$ pairwise orthogonal vectors, and
  \item\label{itm:33} for every two sets $A_1, A_2 \subseteq \calA$ with $|A_1|=|A_2|=2k-1$, there exist vectors $v_1 \in A_1$ and $v_2 \in A_2$ with $\langle v_1,v_2 \rangle = 0$.
\end{enumerate}
\end{theorem}

We observe that Theorem~\ref{thm:General} yields Theorems~\ref{thm:IntroBeta} and~\ref{thm:IntroAlpha}.
As mentioned earlier, to prove these theorems, it suffices to consider the field $\Fset_p$ of order $p$.
For Theorem~\ref{thm:IntroBeta}, apply Theorem~\ref{thm:General} with $k$ being $\lfloor \frac{k}{2} \rfloor+1$ and with $\ell=1$ to obtain, using Item~\ref{itm:33}, the desired bound on $\beta(d,k,\Fset_p)$ for an appropriate $\delta = \delta(p)$.
For Theorem~\ref{thm:IntroAlpha}, apply Theorem~\ref{thm:General} with $k$ and $\ell$ being $k+1$ and $\ell+1$ respectively to obtain, using Item~\ref{itm:22}, the desired bound on $\alpha(d,k,\ell,\Fset_p)$.

Towards the proof of Theorem~\ref{thm:General}, we apply the results from the previous section to a family of graphs, defined next.

\subsection{The Orthogonality Graph}

For a prime $p$ and an integer $t$, let $G(p,t)$ denote the graph whose vertices are all the nonzero vectors in $\Fset_p^t$, where two such (not necessarily distinct) vectors are adjacent if and only if they are orthogonal.
The second largest eigenvalue of $G(p,t)$ was determined in~\cite{AlonK97,Vinh08a}, as stated below.

\begin{proposition}[{\cite{AlonK97,Vinh08a}}]\label{prop:G(p,t)}
For every prime $p$ and every integer $t$, the graph $G(p,t)$ is an $(n,d,\lambda)$-graph for
\[n=p^t-1, ~~d= p^{t-1}-1, \mbox{~~and~~} \lambda = (p-1) \cdot p^{t/2-1}.\]
\end{proposition}

By applying Corollary~\ref{cor:ARbipartite<=} to the graph $G(p,t)$, we obtain the following result.

\begin{theorem}\label{thm:bipartite-free}
For every prime $p$, there exists a constant $c=c(p)$, such that for all integers $t$ and $k$, the number of pairs $(C_1,C_2)$ of subsets of $\Fset_p^t \setminus \{0\}$ with $|C_1| \leq k$ and $|C_2| \leq k$, such that $\langle v_1, v_2 \rangle \neq 0$ for all $v_1 \in C_1$ and $v_2 \in C_2$, is at most $2^{c \cdot (t^2+k)} \cdot p^{t \cdot k}$.
\end{theorem}

\begin{proof}
Fix a prime $p$, and let $t$ and $k$ be some integers.
By Proposition~\ref{prop:G(p,t)}, the graph $G(p,t)$ is an $(n,d,\lambda)$-graph for $n=p^t-1$, $d= p^{t-1}-1$, and $\lambda = (p-1) \cdot p^{t/2-1} \leq p^{t/2}$.
Letting $s = \frac{2n \log n}{d}$, it holds that $s = \Theta(t)$, and it is not difficult to verify that $(\frac{2 \lambda n}{d})^2 \geq n$.
By Corollary~\ref{cor:ARbipartite<=}, the number of pairs $(C_1,C_2)$ of sets of vertices of $G(p,t)$ with $|C_1| \leq k$ and $|C_2| \leq k$, such that no edge connects a vertex of $C_1$ to a vertex of $C_2$, is at most
\begin{eqnarray*}
(k+1)^2 \cdot \Big (\frac{2 \lambda n}{d} \Big )^{2 \cdot (s+k)} & = & (k+1)^2 \cdot \Big (\frac{2 n}{d} \Big )^{2 \cdot (s+k)} \cdot \lambda^{2s} \cdot \lambda^{2k} \\ &\leq& 2^{O(k)} \cdot 2^{O(s+k)} \cdot 2^{O(s \cdot t)} \cdot p^{t \cdot k} \leq 2^{O(t^2 +k)} \cdot p^{t \cdot k}.
\end{eqnarray*}
By the definition of the graph $G(p,t)$, the proof is complete.
\end{proof}

By applying Theorem~\ref{thm:F-free} to the graph $G(p,t)$, we obtain the following result.

\begin{theorem}\label{thm:Kell-free}
For every prime $p$ and every integer $\ell \geq 2$, there exists a constant $c=c(p,\ell)$, such that for all integers $t$ and $k$, the number of subsets of $\Fset_p^t \setminus \{0\}$ of size at most $k$ that include no $\ell$ pairwise orthogonal vectors is at most $2^{c \cdot t^2} \cdot p^{t \cdot k/2}$.
\end{theorem}

\begin{proof}
Fix a prime $p$ and an integer $\ell \geq 2$, and let $t$ and $k$ be some integers.
It may be assumed that $t$ is sufficiently large, because if $t$ is bounded by some constant, then so is $k$, and the statement of the theorem trivially holds with an appropriate constant $c$.
By Proposition~\ref{prop:G(p,t)}, the graph $G(p,t)$ is an $(n,d,\lambda)$-graph for $n=p^t-1$, $d= p^{t-1}-1$, and $\lambda = (p-1) \cdot p^{t/2-1} \leq p^{t/2}$.
Note that $d \leq 0.9 \cdot n$, and that for a growing $t$, we have $n = \Theta(d)$ and $n = \omega(\lambda)$. By Theorem~\ref{thm:F-free}, the number of sets of at most $k$ vertices of $G(p,t)$ with no copy of $K_\ell$ is at most
\[2^{O(\log^2 n)} \cdot \lambda^k \leq 2^{O(t^2)} \cdot p^{t \cdot k/2}.\]
By the definition of the graph $G(p,t)$, the proof is complete.
\end{proof}

\subsection{Proof of Theorem~\ref{thm:General}}

Before turning to the proof of Theorem~\ref{thm:General}, let us collect a few notations and facts about the tensor product operation on vectors, which plays a central role in the argument.
For a field $\Fset$ and integers $t_1,t_2$, the tensor product $w = u \otimes v$ of two vectors $u \in \Fset^{t_1}$ and $v \in \Fset^{t_2}$ is defined as the vector in $\Fset^{t_1 \cdot t_2}$, whose coordinates are indexed by the pairs $(i_1,i_2)$ with $i_1 \in [t_1]$ and $i_2 \in [t_2]$, defined by $w_{(i_1,i_2)} = u_{i_1} \cdot v_{i_2}$. Note that for integers $t$ and $m$ and for given vectors $v_1, \ldots, v_m \in \Fset^t$, the vector $v_1 \otimes \cdots \otimes v_m$ lies in $\Fset^{t^m}$ and consists of all the $t^m$ possible products of $m$ values, one from each vector $v_j$ with $j \in [m]$.
It is well known and easy to verify that for vectors $u_1, \ldots, u_m \in \Fset^t$
and $v_1, \ldots, v_m \in \Fset^t$, the two vectors $u = u_1 \otimes \cdots \otimes u_m$ and $v = v_1 \otimes \cdots \otimes v_m$ satisfy
\begin{eqnarray}\label{eq:tensor}
\langle u, v \rangle = \prod_{j=1}^{m}{\langle u_j , v_j \rangle}.
\end{eqnarray}

\begin{proof}[ of Theorem~\ref{thm:General}]
Let $p$ be a fixed prime, and let $\ell \geq 2$ be a fixed integer.
For integers $t$ and $m$, let $Q \subseteq (\Fset_p^t)^m$ denote the collection of all $m$-tuples of non-self-orthogonal vectors of $\Fset_p^t$.
Notice that the number of non-self-orthogonal vectors in $\Fset^t_p$ is at least $p^{t-1}$, because any choice for the first $t-1$ entries of a vector in $\Fset^t_p$ can be extended to a non-self-orthogonal vector by choosing for its last entry either $0$ or $1$. This implies that $|Q| \geq p^{m \cdot (t-1)}$.

We apply the probabilistic method. For an integer $n$, let $\calZ = (z_1, \ldots, z_n)$ be a random sequence of $n$ elements chosen uniformly and independently from $Q$, and let $\calA \subseteq \Fset_p^{t^m}$ be the set of all $m$-fold tensor products of the $m$-tuples of $\calZ$, that is, the vectors $v_1 \otimes \cdots \otimes v_m$ for which $z_i = (v_1, \ldots, v_m)$ for some $i \in [n]$.
The vectors of $\calA$ are non-self-orthogonal, because for every $(v_1, \ldots, v_m) \in Q$, it follows from~\eqref{eq:tensor} that the vector $v = v_1 \otimes  \cdots \otimes v_m$ satisfies $\langle v,v \rangle = \prod_{i=1}^{m}{\langle v_i,v_i \rangle} \neq 0$.
We will show that for a given integer $k$ and for an appropriate choice of the integers $t$, $m$, and $n$, the set $\calA$ satisfies with positive probability the properties declared in the theorem.

Let $\calC_1$ denote the collection of all (non-empty) subsets of $\Fset_p^t \setminus \{0\}$ of size at most $k$ that include no $\ell$ pairwise orthogonal vectors.
Consider the collection
\[ \calB_1 = \{ C^{(1)} \times C^{(2)} \times \cdots \times C^{(m)} \mid C^{(j)} \in \calC_1 \mbox{~for all~}j \in [m] \}.\]
Notice that each set $B \in \calB_1$ consists of at most $k^m$ $m$-tuples of vectors in $\Fset_p^t$.
Let $\calE_1$ denote the event that some set of $\calB_1$ includes at least $k$ elements of the sequence $\calZ$.
More formally, we define $\calE_1$ as the event that there exist sets $B \in \calB_1$ and $I \subseteq [n]$ with $|I|=k$, such that $\{z_i \mid i \in I\} \subseteq B$.
By the union bound, it follows that
\begin{eqnarray}\label{eq:E1}
\Prob{}{\calE_1} \leq |\calB_1| \cdot \binom{n}{k} \cdot \bigg ( \frac{k^m}{|Q|} \bigg)^k \leq |\calB_1| \cdot \bigg ( \frac{n \cdot k^m}{p^{m \cdot (t-1)}} \bigg)^k.
\end{eqnarray}

Let $\calC_2$ denote the collection of all pairs $(C_1,C_2)$ of (non-empty) subsets of $\Fset_p^t \setminus \{0\}$ of size at most $k$, such that no vector of $C_1$ is orthogonal to a vector of $C_2$. Consider the collection $\calB_2$ of all pairs $(B_1,B_2)$ of the form
\begin{eqnarray}\label{eq:B1B2}
B_1 = C_1^{(1)} \times C_1^{(2)} \times \cdots \times C_1^{(m)} \mbox{~~and~~} B_2 = C_2^{(1)} \times C_2^{(2)} \times \cdots \times C_2^{(m)},
\end{eqnarray}
where $(C^{(1)}_1,C^{(1)}_2), (C^{(2)}_1,C^{(2)}_2), \ldots, (C^{(m)}_1,C^{(m)}_2)$ are $m$ pairs of the collection $\calC_2$.
As before, each $B_1$ and $B_2$ consists of at most $k^m$ $m$-tuples of vectors in $\Fset_p^t$.
Let $\calE_2$ denote the event that there exist a pair $(B_1,B_2) \in \calB_2$ and disjoint sets $I_1,I_2 \subseteq [n]$ with $|I_1|=|I_2|=k$, such that $\{z_i \mid i \in I_1\} \subseteq B_1$ and $\{z_i \mid i \in I_2\} \subseteq B_2$.
By the union bound, it follows that
\begin{eqnarray}\label{eq:E2}
\Prob{}{\calE_2} \leq |\calB_2| \cdot \binom{n}{k}^2 \cdot \bigg ( \frac{k^m}{|Q|} \bigg)^{2k} \leq |\calB_2| \cdot \bigg ( \frac{n \cdot k^m}{p^{m \cdot (t-1)}} \bigg)^{2k}.
\end{eqnarray}

Let us set the parameters of the construction, ensuring that the event $\calE_1 \vee \calE_2$ occurs with probability smaller than $1$.
Let $d$ and $k$ be two integers satisfying $d \geq k \geq \ell$.
We may assume, whenever needed, that $k$ is sufficiently large, because constant values of $k$ can be handled by an appropriate choice of the constant $\delta$ from the assertion of the theorem, using the $d$ vectors of the standard basis of $\Fset_p^d$.
By Theorems~\ref{thm:Kell-free} and~\ref{thm:bipartite-free}, there exists a constant $c \geq 1$, such that $|\calC_1| \leq 2^{c \cdot t^2} \cdot p^{t \cdot k/2}$ and $|\calC_2| \leq 2^{c \cdot (t^2+k)} \cdot p^{t \cdot k}$, implying that
\begin{eqnarray}\label{eq:|calB|}
|\calB_1| = |\calC_1|^m \leq 2^{c m \cdot t^2} \cdot p^{m t k/2}~~\mbox{and}~~|\calB_2| = |\calC_2|^m \leq 2^{c m \cdot (t^2+k)} \cdot p^{m t k}.
\end{eqnarray}
Let $t$ be the largest integer satisfying, say, $k \geq 5c \cdot t$, and let $m$ be the largest integer satisfying $d \geq t^m$.
By the assumption $d \geq k$, we have $m \geq 1$.
Set $n = \lfloor p^{m \cdot t/4} \rfloor$.
Combining~\eqref{eq:E1} and~\eqref{eq:|calB|}, we obtain that
\[ \Prob{}{\calE_1} \leq 2^{c m \cdot t^2} \cdot p^{mtk/2} \cdot \bigg ( \frac{n \cdot k^m}{p^{m \cdot (t-1)}} \bigg )^k
 \leq  2^{c m \cdot t^2} \cdot \bigg ( \frac{k^m}{p^{m \cdot (t/4-1)}} \bigg )^k
 \leq  \bigg ( 2^{t/5} \cdot \frac{5c(t+1)}{p^{t/4-1}} \bigg )^{mk} < \frac{1}{2},\]
where the second inequality holds by our choice of $n$, the third by our choice of $t$, and the fourth by the assumption that $k$ (and thus $t$) is sufficiently large. By a similar calculation, combining~\eqref{eq:E2} and~\eqref{eq:|calB|}, we obtain that
\[ \Prob{}{\calE_2} \leq 2^{c m \cdot (t^2+k)} \cdot p^{mtk} \cdot \bigg ( \frac{n \cdot k^{m}}{p^{m \cdot (t-1)}} \bigg )^{2k}
\leq 2^{2c m \cdot t^2} \cdot \bigg ( \frac{k^m}{p^{m \cdot (t/4-1)}} \bigg )^{2k} <\frac{1}{2},\]
where for the second inequality we further use the inequality $k \leq t^2$, which holds assuming that $k$ is sufficiently large.
It thus follows, by the union bound, that the probability that the event $\calE_1 \vee \calE_2$ occurs is smaller than $1$.
This implies that there exists a choice for the sequence $\calZ$ for which the event $\calE_1 \vee \calE_2$ does not occur.
We fix such a choice for $\calZ$ and consider the corresponding set $\calA$.

We show now that the set $\calA$ satisfies the required properties.
We start by proving that every set $A \subseteq \calA$ with $|A|=k$ includes $\ell$ pairwise orthogonal vectors.
To do so, we show that for every set $I \subseteq [n]$ with $|I|=k$, the $m$-fold tensor products associated with the $m$-tuples of $\{z_i \mid i \in I\}$ include $\ell$ pairwise orthogonal vectors.
So assume for contradiction that there exists a set $I \subseteq [n]$ with $|I|=k$ that does not satisfy this property.
For each $j \in [m]$, let $C^{(j)}$ denote the set of the $j$th projections of the tuples of $\{z_i \mid i \in I\}$, and notice that $|C^{(j)}| \leq k$. Using the property of tensor product given in~\eqref{eq:tensor}, it follows that $C^{(j)}$ does not include $\ell$ pairwise orthogonal vectors. Therefore, the set $C^{(1)} \times C^{(2)} \times \cdots \times C^{(m)}$ lies in $\calB_1$ and contains $\{z_i \mid i \in I\}$. This contradicts the fact that the event $\calE_1$ does not occur for our choice of $\calZ$.

We next prove the second statement of the theorem.
Assume for contradiction that there exist two sets $A_1, A_2 \subseteq \calA$ with $|A_1|=|A_2|=2k-1$, such that no vector of $A_1$ is orthogonal to a vector of $A_2$.
If $|A_1 \cap A_2| \geq k$, then there exists a set of $k$ vectors of $\calA$ with no orthogonal pair, in contradiction to the property of $\calA$ shown above.
Otherwise, there exist disjoint sets $A'_1 \subseteq A_1 \setminus A_2$ and $A'_2 \subseteq A_2 \setminus A_1$ satisfying $|A'_1|=|A'_2|=k$.
Let $I_1,I_2 \subseteq [n]$ be sets with $|I_1|=|I_2|=k$, such that the vectors of $A'_1$ and $A'_2$ are the $m$-fold tensor products associated with the $m$-tuples of $\{z_i \mid i \in I_1\}$ and $\{z_i \mid i \in I_2\}$ respectively. Note that $I_1$ and $I_2$ are disjoint.
For each $j \in [m]$, let $C_1^{(j)}$ and $C_2^{(j)}$ denote the sets of the $j$th projections of the tuples of $\{z_i \mid i \in I_1\}$ and $\{z_i \mid i \in I_2\}$ respectively, and notice that $|C_1^{(j)}| \leq k$ and $|C_2^{(j)}| \leq k$.
Using the property of tensor product given in~\eqref{eq:tensor}, it follows that no vector of $C_1^{(j)}$ is orthogonal to a vector of $C_2^{(j)}$. Therefore, there exists a pair $(B_1,B_2) \in \calB_2$, defined as in~\eqref{eq:B1B2}, for which it holds that $\{z_i \mid i \in I_1\} \subseteq B_1$ and $\{z_i \mid i \in I_2\} \subseteq B_2$.
This contradicts the fact that the event $\calE_2$ does not occur for our choice of $\calZ$.

We finally analyze the size of the collection $\calA$.
Recall that the vectors of $\calA$ are non-self-orthogonal.
It follows from the above discussion that no vector of $\calA$ is associated with more than $k-1$ of the $m$-tuples of $\calZ$.
This implies that
\[ |\calA| \geq \frac{n}{k-1} \geq p^{\Omega(m \cdot t)} \geq p^{\Omega((\log d) \cdot t / (\log t))} \geq d^{\Omega(t/\log t)} \geq d^{\Omega(k/\log k)},\]
where the multiplicative constants hidden by the $\Omega$ notation depend only on $p$ and $\ell$.
By adding $d-t^m$ zero entries at the end of each vector of $\calA$, we obtain the desired subset of $\Fset_p^d$, and the proof is complete.
\end{proof}

\section*{Acknowledgments}
We thank the anonymous referees for their constructive and helpful feedback.

\bibliographystyle{abbrv}
\bibliography{larger_nearly}

\end{document}